\documentclass[12pt, a4paper]{amsart}
\usepackage{amssymb}
\usepackage{amsthm}
\usepackage{amsmath}
\usepackage{amsfonts}
\usepackage{amsthm}
\usepackage{esint}
\usepackage{hyperref}
\usepackage{textcomp}
\usepackage{graphicx}
\usepackage{cite}
\usepackage{stmaryrd}
\usepackage{mathtools}

\newtheorem{theorem}{Theorem}
\newtheorem{lemma}[theorem]{Lemma}
\newtheorem{corollary}[theorem]{Corollary}
\DeclareMathOperator{\Mod}{{\rm mod}\,}

\begin{document}
	
	\title{Refined Upper Bounds for $L(1,\chi)$}
	\author{Jeffery Ezearn}
	\address{Department of Mathematics, College of Science, Kwame Nkrumah University of Science and Technology, Kumasi, Ghana}
	\email{jezearn@knust.edu.gh}
	\maketitle
	
	\begin{abstract}
	Let $\chi$ be a non-principal Dirichlet character of modulus $q$ with associated \textit{L}-function $L(s,\chi)$. We prove that
	$$|L(1,\chi)|\le\left(\frac{1}{2}+O\Big(\frac{\log\log q}{\log q}\Big)\right)\frac{\varphi(q)}{q}\log q\,,$$
	where $\varphi(\cdot)$ is Euler's phi function. This refines known bounds of the form $
	(c+o(1))\log q $ or $(c+O(\frac{1}{\log q}))\log q $ and is relevant for prime-rich moduli. It follows from Mertens' third theorem and the prime number theorem that $\inf_{q>2}\max_{\chi\ne\chi_0\,(\Mod q)}\frac{|L(1,\chi)|}{\log q/\log\log q}\le\frac{1}{2}e^{-\gamma}$.
\end{abstract}

\section{Introduction}
	Throughout, $\chi(\cdot)$ is a Dirichlet character of modulus $q$ with associated \textit{L}-series $L(s,\chi):=\sum_{n\ge1}\frac{\chi(n)}{n^s}$ and, for distinction, we let $\chi_0$ be the principal character. The focus of this paper is the following theorem.
	
	\begin{theorem}
		\label{L(1,chi) Refined}
		For all non-principal Dirichlet characters $\chi$ modulo $q$,
		$$|L(1,\chi)|\le\left(\frac{1}{2}+O\Big(\frac{\log\log q}{\log q}\Big)\right)\frac{\varphi(q)}{q}\log q\,,$$
		where the implied constant is independent of $q$; consequently,
		$$\liminf_{q\to\infty}\max_{\chi\ne\chi_0(\Mod q)}\frac{|L(1,\chi)|}{\log q/\log\log q}\le\frac{1}{2}e^{-\gamma}\,.$$
	\end{theorem}
	\noindent
	Here, $\varphi(\cdot)$ is Euler's totient function and $\gamma$ is Euler's constant satisfying $\sum_{1\le k\le x}\frac{1}{k}=\log x+\gamma+O(\frac{1}{x})$. The above theorem refines well-known best bounds of the form $(c+o(1))\log q$, known uniformly for $c=\frac{1}{2}$ and for various values of $c\ge\frac{1}{2}(1-\frac{1}{\sqrt{e}}) $ depending on the prime factorisation of $q$ or the nature\textemdash primitivity, quadraticity\textemdash of $\chi$\footnote{Indeed, as our proof shows, the $\frac{1}{2}$ in Theorem \ref{L(1,chi) Refined} could be replaced accordingly with the known constants of $c$ in these specialisations} (\cite{Chowla,Burgess,Stephens,Pintz,Granville-Soundararajan}). For prime-rich moduli such that $\frac{\varphi(q)}{q}\log\log q=O(1)$, this refinement even extends to the known upper bounds of the form $(c+O(\frac{1}{\log q}))\log q$ (\cite{Louboutin-1,Louboutin-2, Platt-Eddin, Ramare-1, Ramare-2}).
	
	Our proof of Theorem \ref{L(1,chi) Refined} requires the following easy lemma, along with the Polya-Vinogradov inequality, Mertens' theorems, and the prime number theorem. Let $\tau(\cdot)$ be the number of divisors function. Throughout, $\sum_p$ and $\prod_p$ signifies evaluations over distinct prime numbers $p$. 
	
	\begin{lemma}
		\label{Principal L(1,X)}
		For all $x\ge1$, the following identity holds.
		$$\sum_{\substack{1\le n\le x\\\gcd(n,q)=1}}\frac{1}{n}=\frac{\varphi(q)}{q}\left(\log x+\sum_{p|q}\frac{\log p}{p-1}+\gamma\right)+O\Big(\frac{\tau(q)}{x}\Big)\,,$$
		where the implied constant is independent of $q$ or $x$. Furthermore,
		$$\sum_{p|q}\frac{\log p}{p-1}\le\log\log\Big(\prod_{p|q}p\Big)+O(1)\,,$$
		where the implied constant is independent of $q$.
	\end{lemma}
		
	\begin{proof}
		We have, via M\"{o}bius inversion,
		\begin{align*}
			\sum_{\substack{1\le n\le x\\\gcd(n,q)=1}}\frac{1}{n}&=\sum_{d|q}\mu(d)\sum_{\substack{1\le n\le x\\d|n}}\frac{1}{n}\\
			&=\sum_{d|q}\frac{\mu(d)}{d}\sum_{1\le k\le \frac{x}{d}}\frac{1}{k}\\
			&=\sum_{d|q}\frac{\mu(d)}{d}\left(\log\Big(\frac{x}{d}\Big)+\gamma+O\Big(\frac{d}{x}\Big)\right)\\
			&=\frac{\varphi(q)}{q}(\log x+\gamma)-\sum_{d|q}\frac{\mu(d)}{d}\log d+O\Big(\frac{\tau(q)}{x}\Big)\\
			&=\frac{\varphi(q)}{q}\left(\log x+\gamma+\sum_{p|q}\frac{\log p}{p-1}\right)+O\Big(\frac{\tau(q)}{x}\Big)\,.
		\end{align*}
		For the estimate of $\sum_{p|q}\frac{\log p}{p-1}$, we invoke Mertens' second theorem \cite{Mertens}, essentially $\sum_{p\le t}\frac{\log p}{p-1}=\log t+O(1) $, where the implied constant is independent of $t$. Let $q_0=\prod_{p|q}p$; then,
		\begin{align*}
			\sum_{p|q}\frac{\log p}{p-1}&=\sum_{p|q_0,\,p\le\log q_0}\frac{\log p}{p-1}+\sum_{p|q_0,\,p>\log q_0}\frac{\log p}{p-1}\\
			&\le\log\log q_0+O(1)+\sum_{p|q_0,\,p>\log q}\frac{\log p}{\log q_0}\\
				&\le\log\log q_0+O(1)+\frac{\log q_0}{\log q_0}\,,
		\end{align*}
		as was to be demonstrated.
	\end{proof}
	\noindent
	A quick corollary to Lemma \ref{Principal L(1,X)} is the following weaker version of Theorem \ref{L(1,chi) Refined}. However, we present an alternative proof using the digamma function, which allows us to circumvent the error term $O\Big(\frac{\tau(q)}{x}\Big)$ above.
	\begin{corollary}
		\label{L(1,chi) Initial}
		For all non-principal Dirichlet characters $\chi$ modulo $q$,
		$$|L(1,\chi)|<\frac{\varphi(q)}{q}\left(\log q+\sum_{p|q}\frac{\log p}{p-1}\right)\,,$$
		In particular, $|L(1,\chi)|\le\left(1+O\Big(\frac{\log\log q}{\log q}\Big)\right)\frac{\varphi(q)}{q}\log q$ and, consequently,
		$$\liminf_{q\to\infty}\max_{\chi\ne\chi_0(\Mod q)}\frac{|L(1,\chi)|}{\log q/\log\log q}\le e^{-\gamma}\,.$$
	\end{corollary}
	\begin{proof}
		Let $\psi(z)=-\gamma-\sum_{n\ge0}\frac{1-z}{(n+z)(n+1)}$ be the digamma function. Recall that
		\begin{align*}
			L(1,\chi)&=\lim_{s\searrow1}\frac{1}{q^s}\sum_{k=1}^q\chi(k)\sum_{n\ge0}\frac{1}{(n+\frac{k}{q})^s}\\
			&=\lim_{s\searrow1}\frac{1}{q^s}\sum_{k=1}^q\chi(k)\sum_{n\ge0}\left(\frac{1}{(n+\frac{k}{q})^s}-\frac{1}{(n+1)^s}\right)\\
			&=\frac{1}{q}\sum_{k=1}^q\chi(k)\left(-\psi\Big(\frac{k}{q}\Big)-\gamma\right)\,,
		\end{align*}
	where, in the second step, one uses the identity $\sum_{k=1}^q\chi(k)=0$. From Gauss' well-known identity $\sum_{k=1}^m\psi\Big(\frac{k}{m}\Big)=-m(\gamma+\log m)$, we obtain
	\begin{align*}
		|L(1,\chi)|&<\frac{1}{q}\sum_{\substack{1\le k\le q\\\gcd(k,q)=1}}\left(-\psi\Big(\frac{k}{q}\Big)-\gamma\right)\\
		&=\frac{1}{q}\sum_{d|q}\mu(d)\sum_{\substack{1\le k\le q\\d|k}}\left(-\psi\Big(\frac{k}{q}\Big)-\gamma\right)\\
		&=\frac{1}{q}\sum_{d|q}\mu(d)\left(\frac{q}{d}\log\Big(\frac{q}{d}\Big)\right)\\
		&=\frac{\varphi(q)}{q}\log q-\sum_{d|q}\frac{\mu(d)}{d}\log d\\
		&=\frac{\varphi(q)}{q}\left(\log q+\sum_{{\rm prime }\,p|q}\frac{\log p}{p-1}\right)\,.
	\end{align*}
	The inferior limit follows by considering the moduli $q(x)=\prod_{p\le x}p$ as $x\to\infty$ and invoking Mertens' third theorem \cite{Mertens}, viz. $\frac{\varphi(q(x))}{q(x)}\sim\frac{1}{e^\gamma\log x}$ and the prime number theorem, viz. $\log\log q(x)\sim \log x$.
	\end{proof}
	
	\subsection*{Proof of Theorem \ref{L(1,chi) Refined}}
		First, recall that
		\begin{align*}
			L(1,\chi)&=\sum_{1\le n\le x}\frac{\chi(n)}{n}+\lim_{s\searrow1}\sum_{n>x}\left(\frac{1}{n^s}-\frac{1}{(n+1)^s}\right)\sum_{x<k\le n}\chi(k)\\
			&=\sum_{1\le n\le x}\frac{\chi(n)}{n}+\sum_{n>x}\frac{1}{n(n+1)}\sum_{x<k\le n}\chi(k)\,,
		\end{align*}
		from which it follows that
		\begin{equation}
			\label{Eq 1}
			|L(1,\chi)|\le\sum_{\substack{1\le n\le x\\\gcd(n,q)=1}}\frac{1}{n}+\frac{1}{x}\sup_{n>m}\left|\sum_{x<k\le n}\chi(k)\right|\,.
		\end{equation}
		\noindent
		Now, recall P\'{o}lya-Vinogradov's estimate (for instance, \cite[Ch. 9]{Montgomery-Vaughan}),
		$$\left|\sum_{x<k\le n}\chi(k)\right|=O(\sqrt{q}\log q)$$
		for any non-principal Dirichlet character $\chi$ modulo $q$ and $x>0$, and choose $x=\sqrt{q}\log q$. Then from (\ref{Eq 1}) and Lemma \ref{L(1,chi) Refined}, we obtain
	
		\begin{align*}
			|L(1,\chi)|&\le\frac{\varphi(q)}{q}\left(\log(\sqrt{q}\log q)+\gamma+\sum_{p|q}\frac{\log p}{p-1}\right)+O\Big(\frac{\tau(q)}{\sqrt{q}\log q}\Big)+O(1)\\
			&\le\frac{\varphi(q)}{q}\Big(\log(\sqrt{q}\log q)+O(\log\log q)\Big)+O(1)\\
			&=\frac{\varphi(q)}{q}\left(\log(\sqrt{q}\log q)+O(\log\log q)+O\Big(\frac{q}{\varphi(q)}\Big)\right)\\
			&=\frac{\varphi(q)}{q}\left(\frac{1}{2}\log q+O(\log\log q)\right)\,,
		\end{align*}
		where, in the last step, we have used the well-known fact (and, in any case, easy to prove via Mertens' third theorem) that $\frac{q}{\varphi(q)}=O(\log\log q)$. The inferior limit in the theorem then follows by considering the moduli $q:=\prod_{p\le x}p$ as $x\to\infty$ and invoking Mertens' third theorem and the prime number theorem.

\end{document}